\documentclass{article}%
\usepackage{amssymb}
\usepackage{amsmath}
\usepackage{amsfonts}
\usepackage{graphicx}%
\providecommand{\U}[1]{\protect\rule{.1in}{.1in}}
\newtheorem{theorem}{Theorem}[section]

\newtheorem{conjecture}[theorem]{Conjecture}
\newtheorem{corollary}[theorem]{Corollary}

\newtheorem{lemma}[theorem]{Lemma}

\newtheorem{problem}[theorem]{Problem}
\newtheorem{proposition}[theorem]{Proposition}

\newenvironment{proof}[1][Proof]{\noindent\textbf{#1.} }{\ \rule{0.5em}{0.5em}}
\begin{document}

\author{Vadim E. Levit\\Department of Mathematics\\Ariel University, Israel\\levitv@ariel.ac.il
\and Eugen Mandrescu\\Department of Computer Science\\Holon Institute of Technology, Israel\\eugen\_m@hit.ac.il}
\title{On Minimal Critical Independent Sets of Almost Bipartite
non-K\"{o}nig-Egerv\'{a}ry Graphs}
\date{}
\maketitle

\begin{abstract}
A set $S\subseteq V$ is \textit{independent} in a graph $G=\left(  V,E\right)
$ if no two vertices from $S$ are adjacent. The \textit{independence number}
$\alpha(G)$ is the cardinality of a maximum independent set, while $\mu(G)$ is
the size of a maximum matching in $G$. If $\alpha(G)+\mu(G)$ equals the order
of $G$, then $G$ is called a \textit{K\"{o}nig-Egerv\'{a}ry graph
}\cite{dem,ster}. The number $d\left(  G\right)  =\max\{\left\vert
A\right\vert -\left\vert N\left(  A\right)  \right\vert :A\subseteq V\}$ is
called the \textit{critical difference} of $G$ \cite{Zhang} (where $N\left(
A\right)  =\left\{  v:v\in V,N\left(  v\right)  \cap A\neq\emptyset\right\}
$). It is known that $\alpha(G)-\mu(G)\leq d\left(  G\right)  $ holds for
every graph \cite{Levman2011a,Lorentzen1966,Schrijver2003}.

A graph $G$ is \textit{(i)} \textit{unicyclic} if it has a unique cycle,
\textit{(ii)} \textit{almost bipartite} if it has only one odd cycle.

Let $\mathrm{\ker}(G)=\bigcap\left\{  S:S\text{ \textit{is a critical
independent set}}\right\}  $, \textrm{core}$\left(  G\right)  $ be the
intersection of all maximum independent sets, and \textrm{corona}$\left(
G\right)  $ be the union of all maximum independent sets\ of $G$. It is known
that \ $\mathrm{\ker}(G)\subseteq\mathrm{core}(G)$ is true for every graph
\cite{Levman2011a}, while the equality holds for bipartite graphs
\cite{Levman2011b}, and for unicyclic non-K\"{o}nig-Egerv\'{a}ry graphs
\cite{LevMan2014}.

In this paper, we prove that if $G$ is an almost bipartite
non-K\"{o}nig-Egerv\'{a}ry graph, then\emph{ }$\mathrm{\ker}(G)=$
$\mathrm{core}(G)$,\emph{ }$\mathrm{corona}(G)$ $\cup$ $N($\textrm{core}%
$\left(  G\right)  )=V(G)$, and $\left\vert \mathrm{corona}(G)\right\vert
+\left\vert \mathrm{core}(G)\right\vert =2\alpha(G)+1$.


\end{abstract}

\section{Introduction}

Throughout this paper $G=(V,E)$ is a finite, undirected, loopless graph
without multiple edges, with vertex set $V=V(G)$ of cardinality $n\left(
G\right)  $, and edge set $E=E(G)$ of size $m\left(  G\right)  $. If $X\subset
V$, then $G[X]$ is the subgraph of $G$ spanned by $X$. By $G-W$ we mean the
subgraph $G[V-W]$, if $W\subset V(G)$. For $F\subset E(G)$, by $G-F$ we denote
the subgraph of $G$ obtained by deleting the edges of $F$, and we use $G-e$,
if $F$ $=\{e\}$. If $A,B$ $\subset V$ and $A\cap B=\emptyset$, then $(A,B)$
stands for the set $\{e=ab:a\in A,b\in B,e\in E\}$. The neighborhood of a
vertex $v\in V$ is the set $N(v)=\{w:w\in V$ \textit{and} $vw\in E\}$, and
$N(A)=\bigcup\{N(v):v\in A\}$, $N[A]=A\cup N(A)$ for $A\subset V$. By
$C_{n},K_{n}$ we mean the chordless cycle on $n\geq$ $3$ vertices, and
respectively the complete graph on $n\geq1$ vertices. In order to avoid
ambiguity, we use $N_{G}(v)$ instead of $N(v)$, and $N_{G}(A)$ instead of
$N(A)$.

A \textit{cycle} is a trail, where the only repeated vertices are the first
and last ones. The graph $G$ is \textit{unicyclic} if it has a unique cycle.

Let us define the \textit{trace} of a family $%
\mathcal{F}%
$ of sets on the set $X$ as $%
\mathcal{F}%
|_{X}=\{F\cap X:F\in%
\mathcal{F}%
\}$.

A set $S$ of vertices is \textit{independent} if no two vertices from $S$ are
adjacent, and an independent set of maximum size will be referred to as a
\textit{maximum independent set}. The \textit{independence number }of $G$,
denoted by $\alpha(G)$, is the cardinality of a maximum independent
set\textit{\ }of $G$. Let $\Omega(G)=\{S:S$ \textit{is a maximum independent
set of} $G\}$, \textrm{core}$(G)={\displaystyle\bigcap}\{S:S\in\Omega(G)\}$
\cite{levm3}, and \textrm{corona}$(G)={\displaystyle\bigcup}\{S:S\in
\Omega(G)\}$ \cite{BorosGolLev}. Clearly, $\alpha(G)\leq\alpha(G-e)\leq
\alpha(G)+1$ holds for each edge $e$. An edge $e\in E(G)$ is $\alpha
$-\textit{critical} whenever $\alpha(G-e)>\alpha(G)$.

The number $d_{G}(X)=\left\vert X\right\vert -\left\vert N(X)\right\vert $ is
the \textit{difference} of the set $X\subseteq V(G)$, and $d(G)=\max
\{d_{G}(X):X\subseteq V\}$ is called the \textit{critical difference} of $G$.
A set $U\subseteq V(G)$ is \textit{critical} if $d_{G}(U)=d(G)$ \cite{Zhang}.
The number $id(G)=\max\{d_{G}(I):I\in\mathrm{Ind}(G)\}$ is called the
\textit{critical independence difference} of $G$. If $A\subseteq V(G)$ is
independent and $d_{G}(A)=id(G)$, then $A$ is called \textit{critical
independent }\cite{Zhang}. Clearly, $d(G)\geq id(G)$ is true for every graph
$G$. It is known that the equality $d(G)$ $=id(G)$ holds for every graph $G$
\cite{Zhang}.

For a graph $G$, let $\mathrm{\ker}(G)=\bigcap\left\{  S:S\text{ \textit{is a
critical independent set}}\right\}  $.

\begin{theorem}
\label{th3}\emph{(i)} \cite{Levman2011a} $\mathrm{\ker}(G)$ is the unique
minimal critical (independent) set of $G$, and $\mathrm{\ker}(G)\subseteq
\mathrm{core}(G)$ is true for every graph.

\emph{(ii)} \cite{Levman2011b,LevMan2014} If $G$ is a bipartite graph, or a
unicyclic non-K\"{o}nig-Egerv\'{a}ry graph, then $\mathrm{\ker}%
(G)=\mathrm{core}(G)$.
\end{theorem}

A matching (i.e., a set of non-incident edges of $G$) of maximum cardinality
$\mu(G)$ is a \textit{maximum matching} of $G$. It is well-known that
$\lfloor\frac{n\left(  G\right)  }{2}\rfloor+1\leq\alpha(G)+\mu(G)\leq
n\left(  G\right)  $ hold for every graph $G$. If $\alpha(G)+\mu(G)=n\left(
G\right)  $, then $G$ is called a K\"{o}nig-Egerv\'{a}ry graph \cite{dem,ster}%
. Various properties of K\"{o}nig-Egerv\'{a}ry graphs are presented in
\cite{bourhams1,bourpull,levm4,LevMan3,LevMan5}. It is known that every
bipartite graph is a K\"{o}nig-Egerv\'{a}ry\emph{ }graph \cite{eger,koen}.
This class includes also non-bipartite graphs (see, for instance, the graph
$G$ in Figure \ref{Fig5}).\begin{figure}[h]
\setlength{\unitlength}{1cm}\begin{picture}(5,1.8)\thicklines
\multiput(4,0.5)(1,0){7}{\circle*{0.29}}
\multiput(5,1.5)(2,0){2}{\circle*{0.29}}
\multiput(9,1.5)(1,0){2}{\circle*{0.29}}
\put(4,0.5){\line(1,0){6}}
\put(5,0.5){\line(0,1){1}}
\put(7,1.5){\line(1,-1){1}}
\put(7,0.5){\line(0,1){1}}
\put(7,1.5){\line(1,0){2}}
\put(9,0.5){\line(0,1){1}}
\put(9,1.5){\line(1,0){1}}
\put(10,0.5){\line(0,1){1}}
\put(4,0.1){\makebox(0,0){$a$}}
\put(4.7,1.5){\makebox(0,0){$b$}}
\put(6,0.1){\makebox(0,0){$c$}}
\put(3.2,1){\makebox(0,0){$G$}}
\end{picture}
\caption{$G$ is a K\"{o}nig-Egerv\'{a}ry graph with \textrm{core}$(G)=\left\{
a,b,c\right\}  $ and $\mathrm{\ker}(G)=\{a,b\}$.}%
\label{Fig5}%
\end{figure}
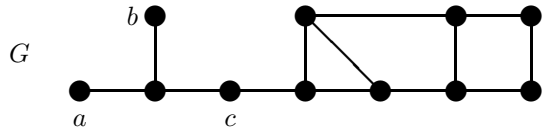

\begin{theorem}
\label{th1} If $G$ is a K\"{o}nig-Egerv\'{a}ry graph, then

\emph{(i) \cite{levm4} }$\mathrm{corona}(G)\cup$ $N(\mathrm{core}(G))=V(G)$;

\emph{(ii) \cite{LevManLemma2014}\ }$\left\vert \mathrm{core}\left(  G\right)
\right\vert +\left\vert \mathrm{corona}\left(  G\right)  \right\vert
=2\alpha(G)$.
\end{theorem}

We call a graph $G$\textit{\ almost bipartite} if it has a unique odd cycle,
denoted $C=\left(  V(C),E\left(  C\right)  \right)  $. Since $C$ is unique, it
is chordless, and there is no other cycle of $G$ sharing edges with $C$. For
every $y\in V(C)$, let us define $D_{y}=(V_{y},E_{y})$ as the connected
bipartite subgraph of $G-E(C)$ containing $y$, and
\[
N_{1}(C)=\{v:v\in V\left(  G\right)  -V(C),N(v)\cap V(C)\neq\emptyset\}.
\]

Clearly, every unicyclic graph with an odd cycle is almost bipartite.

\begin{proposition}
\label{prop3}If $G$\textit{\ is almost bipartite with }$C=\left(
V(C),E\left(  C\right)  \right)  $ as its unique odd cycle\textit{, then
}$V\left(  D_{a}\right)  \cap V\left(  D_{b}\right)  =\emptyset$ for every two
different vertices $a,b\in V(C)$.
\end{proposition}

\begin{proof}
Assume, to the contrary, that there exist $a,b\in V(C)$, such that $V\left(
D_{a}\right)  \cap V\left(  D_{b}\right)  \neq\emptyset$. Let $x\in V\left(
D_{a}\right)  \cap V\left(  D_{b}\right)  $. Thus, there exists some path
containing $x$, and connecting $a$ and $b$. Let $P_{1}$ be a shortest one of
this kind. On the other hand, there exist two paths, say $P_{2}$ and $P_{3}$,
connecting $a$ and $b$, and containing only vertices belonging to $C$.
Therefore, either $P_{1}$ and $P_{2}$, or $P_{1}$ and $P_{3}$, give birth to
an odd cycle, different from $C$, and thus contradicting the fact that $C$ is
the unique odd cycle of $G$.
\end{proof}

As a consequence of Proposition \ref{prop3}, we may infer that $\left\{
V(D_{y}):y\in V(C)\right\}  $ is a partition of $V(G)$.

There exist K\"{o}nig-Egerv\'{a}ry graphs $G$ with $\mathrm{\ker}%
(G)\neq\mathrm{core}(G)$; for instance, the graph in Figure \ref{Fig5}.

There are also almost bipartite K\"{o}nig-Egerv\'{a}ry graph may have
$\mathrm{\ker}(G)\neq\mathrm{core}(G)$; e.g., the graphs in Figure \ref{Fig4}
have \textrm{core}$(G_{1})=\left\{  a\right\}  $ and \textrm{core}%
$(G_{2})=\left\{  u,v,w\right\}  $ .

\begin{figure}[h]
\setlength{\unitlength}{1cm}\begin{picture}(5,1.3)\thicklines
\multiput(2,0)(1,0){4}{\circle*{0.29}}
\multiput(3,1)(1,0){2}{\circle*{0.29}}
\put(2,0){\line(1,0){3}}
\put(3,1){\line(1,0){1}}
\put(3,0){\line(0,1){1}}
\put(4,1){\line(1,-1){1}}
\put(2,0.3){\makebox(0,0){$a$}}
\put(1.2,0.5){\makebox(0,0){$G_{1}$}}
\multiput(7,0)(1,0){6}{\circle*{0.29}}
\multiput(8,1)(1,0){5}{\circle*{0.29}}
\put(7,0){\line(1,0){5}}
\put(8,0){\line(0,1){1}}
\put(9,0){\line(0,1){1}}
\put(9,0){\line(1,1){1}}
\put(10,0){\line(0,1){1}}
\put(11,0){\line(0,1){1}}
\put(11,1){\line(1,0){1}}
\put(12,0){\line(0,1){1}}
\put(7,0.3){\makebox(0,0){$u$}}
\put(8.3,1){\makebox(0,0){$v$}}
\put(9.3,1){\makebox(0,0){$w$}}
\put(6.2,0.5){\makebox(0,0){$G_{2}$}}
\end{picture}\caption{Almost bipartite K\"{o}nig-Egerv\'{a}ry graphs with
$\mathrm{\ker}(G_{1})=\emptyset$ and $\mathrm{\ker}(G_{2})=\{u,v\}$.}%
\label{Fig4}%
\end{figure}
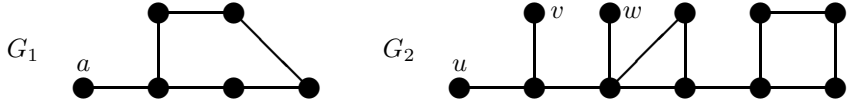

If $H_{j},j=1,2,...,k$, are all the connected components of $G$, \ it is easy
to see that
\begin{align*}
\Omega\left(  G\right)   &  =\bigcup\limits_{j=1}^{k}\Omega\left(
H_{j}\right)  ,\text{\ }\mathrm{core}(G)=\bigcup\limits_{j=1}^{k}%
\mathrm{core}\left(  H_{j}\right)  \text{ , }\\
\mathrm{corona}(G)  &  =\bigcup\limits_{j=1}^{k}\mathrm{corona}\left(
H_{j}\right)  \text{ and }\mathrm{\ker}(G)=\bigcup\limits_{j=1}^{k}%
\mathrm{\ker}\left(  H_{j}\right)  .
\end{align*}

In this paper we show that for every almost bipartite graph $G$, the following hold:

\emph{(i) }$\mathrm{\ker}(G)=$ $\mathrm{core}(G)$;

\emph{(ii) }$\mathrm{corona}(G)$ $\cup$ $N($\textrm{core}$\left(  G\right)
)=V(G)$;

\emph{(iii) }$\left\vert \mathrm{corona}(G)\right\vert +\left\vert
\mathrm{core}(G)\right\vert =2\alpha(G)+1$.

Since $\left\vert \mathrm{corona}(H)\right\vert +\left\vert \mathrm{core}%
(H)\right\vert =2\alpha(H)$ and the assertions \emph{(i)} \ and\ \emph{(ii)}%
\ hold for every bipartite connected component $H$ of $G$, we may assume that
every almost bipartite non-K\"{o}nig-Egerv\'{a}ry graph is connected.

\section{Results}

Recall the following useful results.

\begin{lemma}
\label{lem0}\cite{LevMan2012} For every bipartite graph $H$, a vertex
$v\in\mathrm{core}(H)$ if and only if there exists a maximum matching that
does not saturate $v$.
\end{lemma}

Lemma \ref{lem0} fails for non-bipartite K\"{o}nig-Egerv\'{a}ry graphs; e.g.,
every maximum matching of the graph $G$ from Figure \ref{Fig5} saturates
$c\in$ \textrm{core}$(G)=\{a,b,c\}$.

\begin{lemma}
\cite{LevMan2020}\label{lem2} If $G$ is an almost bipartite graph, then

\emph{(i)} $n(G)-1\leq\alpha(G)+\mu(G)\leq n(G)$;

\emph{(ii)} $n(G)-1=\alpha(G)+\mu(G)$ if and only if each edge of its unique
odd cycle is $\alpha$-critical.
\end{lemma}

\begin{theorem}
\cite{LevMan2020}\label{th2} If $G$ is an almost bipartite
non-K\"{o}nig-Egerv\'{a}ry graph, then

\emph{(i)} \textrm{core}$(G)\cap N\left[  V\left(  C\right)  \right]
=\emptyset$;

\emph{(ii)} $\mathrm{core}\left(  G\right)  =%
{\displaystyle\bigcup\limits_{y\in V(C)}}
\mathrm{core}\left(  D_{y}-y\right)  $;

\emph{(iii)} $\Omega\left(  G\right)  |_{V\left(  D_{y}-y\right)  }%
=\Omega\left(  D_{y}-y\right)  $ for every $y\in V\left(  C\right)  $.
\end{theorem}

\begin{figure}[h]
\setlength{\unitlength}{1cm}\begin{picture}(5,1.3)\thicklines
\multiput(2,0)(1,0){6}{\circle*{0.29}}
\multiput(3,1)(1,0){3}{\circle*{0.29}}
\put(2,0){\line(1,0){5}}
\put(3,0){\line(0,1){1}}
\put(4,0){\line(0,1){1}}
\put(4,1){\line(1,0){1}}
\put(5,1){\line(1,-1){1}}
\put(2,0.3){\makebox(0,0){$a$}}
\put(3.3,1){\makebox(0,0){$b$}}
\put(7,0.3){\makebox(0,0){$c$}}
\put(1.2,0.5){\makebox(0,0){$G_{1}$}}
\multiput(9,0)(1,0){5}{\circle*{0.29}}
\multiput(11,1)(1,0){2}{\circle*{0.29}}
\put(9,0){\line(1,0){4}}
\put(10,0){\line(1,1){1}}
\put(11,1){\line(1,0){1}}
\put(12,0){\line(0,1){1}}
\put(9,0.3){\makebox(0,0){$u$}}
\put(11,0.3){\makebox(0,0){$v$}}
\put(13,0.3){\makebox(0,0){$w$}}
\put(8.2,0.5){\makebox(0,0){$G_{2}$}}
\end{picture}\caption{$G_{1},G_{2}$ are K\"{o}nig-Egerv\'{a}ry graphs,
\textrm{core}$(G_{1})=\left\{  a,b,c\right\}  $, \textrm{core}$(G_{2}%
)=\left\{  u,v,w\right\}  $.}%
\label{fig11222}%
\end{figure}
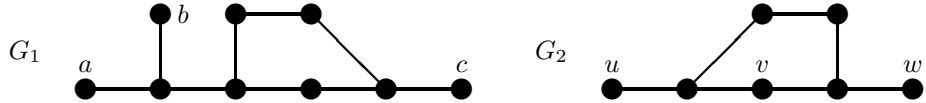

The assertion in Theorem \ref{th2}\emph{(ii)} may fail for connected unicyclic
K\"{o}nig-Egerv\'{a}ry graphs. For instance,
\[
\mathrm{core}\left(  G_{2}\right)  \neq\left\{  u,w\right\}  =%
{\displaystyle\bigcup\limits_{y\in V(C)}}
\mathrm{core}\left(  D_{y}-y\right)  ,
\]
while $\mathrm{core}\left(  G_{1}\right)  =%
{\displaystyle\bigcup\limits_{y\in V(C)}}
\mathrm{core}\left(  D_{y}-y\right)  $, where $G_{1}$ and $G_{2}$ are from
Figure \ref{fig11222}.

\begin{proposition}
\cite{LevMan2020}\label{prop44} Let $G$ be an almost bipartite graph. Then the
following assertions are equivalent:

\emph{(i)} $y\in\mathrm{core}(D_{y})$, for every $y\in V(C)$;

\emph{(ii)} there exists some $S\in\Omega(G)$, such that $S\cap N_{1}%
(C)=\emptyset$;

\emph{(iii)} $n(G)-1=\alpha(G)+\mu(G)$, i.e., $G$ is not a
K\"{o}nig-Egerv\'{a}ry graph.
\end{proposition}

\begin{corollary}
\label{cor3}If $G$ is an almost bipartite non-K\"{o}nig-Egerv\'{a}ry graph,
then there exists some $S\in\Omega(G)$, such that $\left\vert S\cap
V(C)\right\vert =\left\lfloor \frac{\left\vert V(C)\right\vert }%
{2}\right\rfloor $, where $C$ is its unique odd cycle.
\end{corollary}

\begin{lemma}
\label{lem9}If $G$ is an almost bipartite non-K\"{o}nig-Egerv\'{a}ry graph,
then%
\[
\alpha(G)=%
{\textstyle\sum\limits_{y\in V(C)}^{{}}}
\alpha\left(  D_{y}\right)  -\left\lfloor \frac{\left\vert V(C)\right\vert
}{2}\right\rfloor -1,
\]
where $C$ is its unique odd cycle.
\end{lemma}

\begin{proof}
By Corollary \ref{cor3}, there is a maximum independent set $S\in\Omega(G)$
such that $\left\vert S\cap V(C)\right\vert =\left\lfloor \frac{\left\vert
V(C)\right\vert }{2}\right\rfloor $. Therefore, by Proposition \ref{prop44}%
\emph{(i)},
\begin{gather*}
\alpha(G)=%
{\textstyle\sum\limits_{y\in S\cap V(C)}^{{}}}
\alpha\left(  D_{y}\right)  +%
{\textstyle\sum\limits_{y\in V(C)-S}^{{}}}
\left(  \alpha\left(  D_{y}\right)  -1\right) \\
=%
{\textstyle\sum\limits_{y\in V(C)}^{{}}}
\alpha\left(  D_{y}\right)  -\left\vert V(C)-S\right\vert =%
{\textstyle\sum\limits_{y\in V(C)}^{{}}}
\alpha\left(  D_{y}\right)  -\left\lfloor \frac{\left\vert V(C)\right\vert
}{2}\right\rfloor -1,
\end{gather*}
as required.
\end{proof}

\begin{proposition}
\label{prop2}If $G$ is an almost bipartite non-K\"{o}nig-Egerv\'{a}ry graph,
then every maximum matching of $G$ contains at least one edge belonging to its
unique odd cycle.
\end{proposition}

\begin{proof}
Assume, to the contrary, that there exists some maximum matching $M$ of $G$,
such that $M\cap E(C)=\emptyset$.

\textit{Case 1. }There exist two consecutive vertices on $C$, say $y_{1}%
,y_{2}$, such that $D_{y_{1}}=\{y_{1}\}$ and $D_{y_{2}}=\{y_{2}\}$.

Since $G-y_{1}y_{2}$ is a bipartite graph, we have that%
\begin{gather*}
\alpha(G)+\mu(G)+1=n(G)=n(G-y_{1}y_{2})\\
=\alpha(G-y_{1}y_{2})+\mu(G-y_{1}y_{2})=\alpha(G)+1+\mu(G-y_{1}y_{2})
\end{gather*}
which leads to $\mu(G-y_{1}y_{2})=\mu(G)=|M|$. Since $M\cap E(C)=\emptyset$,
we infer that $M\cup\{y_{1}y_{2}\}$ is a matching in $G$, larger than $M$,
contradicting the fact that $\mu(G)=|M|$.

\textit{Case 2. }No two consecutive vertices on $C$, say $y_{1},y_{2}$,
satisfy both $D_{y_{1}}=\{y_{1}\}$ and $D_{y_{2}}=\{y_{2}\}$. It follows that
the number $k$ of vertices $y_{1},y_{2},...,y_{k}$ on $C$ with $D_{y_{i}%
}=\{y_{i}\}$ satisfies $k\leq\left\lfloor \frac{\left\vert V(C)\right\vert
}{2}\right\rfloor $.

Let $y_{k+1},y_{k+2},...,y_{k+p}$ be all the vertices on $C$ with $\left\vert
V(D_{y_{i}})\right\vert =n\left(  D_{y_{i}}\right)  \geq2$. Hence,
$p\geq\left\lfloor \frac{\left\vert V(C)\right\vert }{2}\right\rfloor $.

Since every $D_{y_{i}}$ is bipartite, we know that $n\left(  D_{y_{i}}\right)
=\alpha\left(  D_{y_{i}}\right)  +\mu\left(  D_{y_{i}}\right)  $. In addition,
$\mu(G)=%
{\textstyle\sum\limits_{i=k+1}^{k+p}}
\mu\left(  D_{y_{i}}\right)  $, because $M\cap E(C)=\emptyset$.

Thus%
\[
n(G)=%
{\textstyle\sum\limits_{i=1}^{k+p}}
n\left(  D_{y_{i}}\right)  =%
{\textstyle\sum\limits_{i=1}^{k}}
n\left(  D_{y_{i}}\right)  +%
{\textstyle\sum\limits_{i=k+1}^{k+p}}
n\left(  D_{y_{i}}\right)  =k+%
{\textstyle\sum\limits_{i=k+1}^{k+p}}
n\left(  D_{y_{i}}\right)  .
\]

Consequently, by Proposition \ref{prop44}\emph{(iii)} and Lemma \ref{lem9},%
\begin{align*}
n(G)  &  =\alpha(G)+\mu(G)+1=%
{\textstyle\sum\limits_{y\in V(C)}^{{}}}
\alpha\left(  D_{y}\right)  -\left\lfloor \frac{\left\vert V(C)\right\vert
}{2}\right\rfloor -1+\mu(G)+1\\
=  &
{\textstyle\sum\limits_{y\in V(C)}^{{}}}
\alpha\left(  D_{y}\right)  -\left\lfloor \frac{\left\vert V(C)\right\vert
}{2}\right\rfloor +\mu(G).
\end{align*}

On the other hand, we have%
\[
n(G)=k+%
{\textstyle\sum\limits_{i=k+1}^{k+p}}
n\left(  D_{y_{i}}\right)  =k+%
{\textstyle\sum\limits_{i=k+1}^{k+p}}
\alpha\left(  D_{y_{i}}\right)  +%
{\textstyle\sum\limits_{i=k+1}^{k+p}}
\mu\left(  D_{y_{i}}\right)  =k+%
{\textstyle\sum\limits_{i=k+1}^{k+p}}
\alpha\left(  D_{y_{i}}\right)  +\mu(G).
\]

Hence, we get
\begin{gather*}%
{\textstyle\sum\limits_{y\in V(C)}^{{}}}
\alpha\left(  D_{y}\right)  -\left\lfloor \frac{\left\vert V(C)\right\vert
}{2}\right\rfloor =k+%
{\textstyle\sum\limits_{i=k+1}^{k+p}}
\alpha\left(  D_{y_{i}}\right) \\%
{\textstyle\sum\limits_{i=1}^{k}}
\alpha\left(  D_{y_{i}}\right)  =k+\left\lfloor \frac{\left\vert
V(C)\right\vert }{2}\right\rfloor .
\end{gather*}

Taking into account that $%
{\textstyle\sum\limits_{i=1}^{k}}
\alpha\left(  D_{y_{i}}\right)  =k$ by definition of the sequence $y_{1}%
,y_{2},...,y_{k}$, we arrive at a contradiction.
\end{proof}

Proposition \ref{prop2} is not true for almost bipartite
K\"{o}nig-Egerv\'{a}ry graphs; e.g., the graphs in Figure \ref{Fig 11}.

\begin{figure}[h]
\setlength{\unitlength}{1cm}\begin{picture}(5,1.3)\thicklines
\multiput(2,0)(1,0){6}{\circle*{0.29}}
\multiput(3,1)(1,0){5}{\circle*{0.29}}
\put(2,0){\line(1,0){5}}
\put(3,1){\line(1,0){4}}
\put(3,0){\line(1,1){1}}
\put(5,0){\line(0,1){1}}
\put(1.2,0.5){\makebox(0,0){$G_{1}$}}
\multiput(9,0)(1,0){4}{\circle*{0.29}}
\multiput(9,1)(1,0){2}{\circle*{0.29}}
\put(12,1){\circle*{0.29}}
\put(9,0){\line(1,0){3}}
\put(9,1){\line(1,0){1}}
\put(10,0){\line(0,1){1}}
\put(10,1){\line(1,-1){1}}
\put(11,0){\line(1,1){1}}
\put(8.2,0.5){\makebox(0,0){$G_{2}$}}
\end{picture}\caption{$G_{1}$ and $G_{2}$ are almost bipartite
K\"{o}nig-Egerv\'{a}ry graphs}%
\label{Fig 11}%
\end{figure}
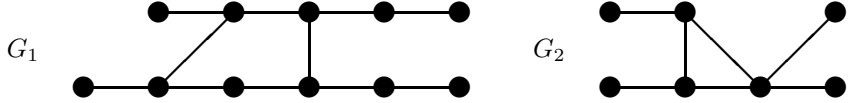

\begin{lemma}
\cite{LevMan2020}\label{lem8} Let $G$\ be an almost bipartite
non-K\"{o}nig-Egerv\'{a}ry graph with the unique odd cycle $C$.

\emph{(i)} If $A$ is a critical independent set, then $A\cap V(C)=\emptyset$.

\emph{(ii)} $\mathrm{core}(G)$ is a critical set.
\end{lemma}

\begin{lemma}
\label{lem5}Let $G$ be an almost bipartite graph. If there is $x\in N_{1}(C)$,
such that $x\in\mathrm{core}(D_{y}-y)$ for some $y\in V(C)$, then $G$ is a
K\"{o}nig-Egerv\'{a}ry graph.
\end{lemma}

\begin{proof}
Let $x\in\mathrm{core}(D_{y}-y)$, $y\in N\left(  x\right)  \cap V(C)$, and
$z\in N\left(  y\right)  \cap V(C)$. Suppose, to the contrary, that $G$ is not
a K\"{o}nig-Egerv\'{a}ry graph. By\emph{ }Lemma \ref{lem2}, the edge $yz$ is
$\alpha$-critical. By Lemma \ref{lem8}, $y\notin\mathrm{core}(G)$. Thus it
follows that $\alpha(G)=\alpha(G-y)$. By Lemma \ref{lem0} there exists a
maximum matching $M_{x}$ of $D_{y}-y$ not saturating $x$. Combining $M_{x}$
with a maximum matching of $G-D_{y}$ we get a maximum matching $M_{y}$ of
$G-y$. Hence $M_{y}\cup\left\{  xy\right\}  $ is a matching of $G$, which
results in $\mu\left(  G\right)  \geq\mu\left(  G-y\right)  +1$. Consequently,
using Lemma \ref{lem2}\emph{(ii)} and having in mind that $G-y$ is a bipartite
graph of order $n(G)-1$, we get the following contradiction
\[
n(G)-1=\alpha(G)+\mu\left(  G\right)  \geq\alpha(G-y)+\mu\left(  G-y\right)
+1=n(G)-1+1=n(G),
\]
and this completes the proof.
\end{proof}

There exist K\"{o}nig-Egerv\'{a}ry and non-K\"{o}nig-Egerv\'{a}ry graphs
having $\mathrm{core}(G)\neq\mathrm{\ker}(G)$; e.g., the graphs from Figure
\ref{Fig8}: $\mathrm{core}\left(  G_{1}\right)  =\left\{  x,y,z\right\}  $ and
$\mathrm{core}\left(  G_{2}\right)  =\left\{  a,b,c\right\}  $.

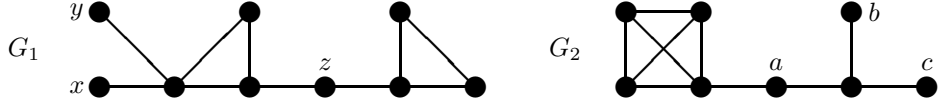
\begin{figure}[h]
\setlength{\unitlength}{1cm}\begin{picture}(5,1.3)\thicklines
\multiput(2,0)(1,0){6}{\circle*{0.29}}
\put(6,1){\circle*{0.29}}
\multiput(2,1)(2,0){2}{\circle*{0.29}}
\put(2,0){\line(1,0){5}}
\put(2,1){\line(1,-1){1}}
\put(3,0){\line(1,1){1}}
\put(4,0){\line(0,1){1}}
\put(6,1){\line(1,-1){1}}
\put(6,0){\line(0,1){1}}
\put(1.7,0){\makebox(0,0){$x$}}
\put(1.7,1){\makebox(0,0){$y$}}
\put(5,0.3){\makebox(0,0){$z$}}
\put(1,0.5){\makebox(0,0){$G_{1}$}}
\multiput(9,0)(1,0){5}{\circle*{0.29}}
\multiput(9,1)(1,0){2}{\circle*{0.29}}
\put(12,1){\circle*{0.29}}
\put(9,0){\line(1,0){4}}
\put(9,0){\line(1,1){1}}
\put(9,1){\line(1,-1){1}}
\put(9,0){\line(0,1){1}}
\put(9,1){\line(1,0){1}}
\put(10,0){\line(0,1){1}}
\put(12,0){\line(0,1){1}}
\put(11,0.3){\makebox(0,0){$a$}}
\put(12.3,1){\makebox(0,0){$b$}}
\put(13,0.3){\makebox(0,0){$c$}}
\put(8.2,0.5){\makebox(0,0){$G_{2}$}}
\end{picture}\caption{ $\mathrm{\ker}(G_{1})=\left\{  x,y\right\}  $,
$\mathrm{\ker}(G_{2})=\left\{  b,c\right\}  $ and only $G_{1}$ is a
K\"{o}nig-Egerv\'{a}ry graph}%
\label{Fig8}%
\end{figure}

\begin{theorem}
\label{th3322}Let $G$\ be an almost bipartite non-K\"{o}nig-Egerv\'{a}ry graph
with the unique odd cycle $C$. Then
\[
\mathrm{\ker}\left(  G\right)  =%
{\displaystyle\bigcup\limits_{y\in V(C)}}
\mathrm{\ker}\left(  D_{y}-y\right)  =%
{\displaystyle\bigcup\limits_{y\in V(C)}}
\mathrm{core}\left(  D_{y}-y\right)  =\mathrm{core}\left(  G\right)  .
\]

\end{theorem}

\begin{proof}
By Theorem \ref{th2}, we have that $\mathrm{core}\left(  G\right)  =%
{\displaystyle\bigcup\limits_{y\in V(C)}}
\mathrm{core}\left(  D_{y}-y\right)  $.

Since every $D_{y}-y$ is a bipartite graph, we infer that $\mathrm{\ker
}\left(  D_{y}-y\right)  =\mathrm{core}\left(  D_{y}-y\right)  $, by Theorem
\ref{th3}\emph{(ii)}.

Consequently, we obtain%

\[
\mathrm{core}\left(  G\right)  =%
{\displaystyle\bigcup\limits_{y\in V(C)}}
\mathrm{core}\left(  D_{y}-y\right)  =%
{\displaystyle\bigcup\limits_{y\in V(C)}}
\mathrm{\ker}\left(  D_{y}-y\right)  .
\]

By Lemma \ref{lem8}\emph{(ii)}, the set $\mathrm{core}\left(  G\right)  $ is
critical in $G$. Hence, we get that
\[
\mathrm{\ker}\left(  G\right)  \subseteq\mathrm{core}\left(  G\right)  =%
{\displaystyle\bigcup\limits_{y\in V(C)}}
\mathrm{\ker}\left(  D_{y}-y\right)  .
\]
Thus it is enough to show that
\[%
{\displaystyle\bigcup\limits_{y\in V(C)}}
\mathrm{\ker}\left(  D_{y}-y\right)  \subseteq\mathrm{\ker}\left(  G\right)
.
\]
In other words, $\mathrm{\ker}\left(  D_{y}-y\right)  $ $\subseteq
\mathrm{\ker}\left(  G\right)  |_{V\left(  D_{y}-y\right)  }$ for every $y\in
V\left(  C\right)  $, which is equivalent to the fact that $\mathrm{\ker
}\left(  G\right)  |_{V\left(  D_{y}-y\right)  }$ is critical in $D_{y}-y$.

By Lemma \ref{lem5}, if $A\subseteq\mathrm{core}\left(  D_{y}-y\right)  $,
then $N_{G}(A)=N_{D_{y}-y}(A)$, since $G$\ is a non-Konig-Egervary almost
bipartite graph. Hence it follows $d_{G}(A)=d_{D_{y}-y}(A)$ for every
$A\subseteq\mathrm{\ker}\left(  D_{y}-y\right)  $. Thus, in accordance with
Theorem \ref{th3}\emph{(i)}, if $A\subset\mathrm{\ker}\left(  D_{y}-y\right)
$, then
\begin{equation}
d_{G}(A)=d_{D_{y}-y}(A)<d_{D_{y}-y}\left(  \mathrm{\ker}\left(  D_{y}%
-y\right)  \right)  =d_{G}\left(  \mathrm{\ker}\left(  D_{y}-y\right)
\right)  . \tag{*}\label{*}%
\end{equation}

Since $\mathrm{\ker}\left(  G\right)  \subseteq%
{\displaystyle\bigcup\limits_{y\in V(C)}}
\mathrm{\ker}\left(  D_{y}-y\right)  $,
\begin{gather*}
d_{G}\left(  \mathrm{\ker}(G)\right)  =d_{G}\left(  \mathrm{\ker}\left(
G\right)  \cap%
{\displaystyle\bigcup\limits_{y\in V(C)}}
\mathrm{\ker}\left(  D_{y}-y\right)  \right) \\
=d_{G}\left(
{\displaystyle\bigcup\limits_{y\in V(C)}}
\left(  \mathrm{\ker}\left(  D_{y}-y\right)  \cap\mathrm{\ker}\left(
G\right)  \right)  \right)  =%
{\displaystyle\sum\limits_{y\in V(C)}}
d_{G}\left(  \mathrm{\ker}\left(  D_{y}-y\right)  \cap\mathrm{\ker}\left(
G\right)  \right)  .
\end{gather*}

If $\mathrm{\ker}\left(  G\right)  \neq%
{\displaystyle\bigcup\limits_{y\in V(C)}}
\mathrm{\ker}\left(  D_{y}-y\right)  $, then
\[
\mathrm{\ker}\left(  G\right)  |_{V\left(  D_{y}-y\right)  }\subset
\mathrm{\ker}\left(  D_{y}-y\right)
\]
for some $y\in V(C)$. Consequently, using the inequality (*) for
$A=\mathrm{\ker}\left(  G\right)  |_{V\left(  D_{y}-y\right)  }$, we obtain
\begin{gather*}
d_{G}\left(  \mathrm{\ker}(G)\right)  =%
{\displaystyle\sum\limits_{y\in V(C)}}
d_{G}\left(  \mathrm{\ker}\left(  G\right)  |_{V\left(  D_{y}-y\right)
}\right) \\
<%
{\displaystyle\sum\limits_{y\in V(C)}}
d_{G}\left(  \mathrm{\ker}\left(  D_{y}-y\right)  \right)  =d_{G}\left(
{\displaystyle\bigcup\limits_{y\in V(C)}}
\mathrm{\ker}\left(  D_{y}-y\right)  \right)  =d\left(  \mathrm{core}\left(
G\right)  \right)  =d(G),
\end{gather*}
which stays in contradiction with the fact that $\mathrm{\ker}(G)$ is critical
in $G$.
\end{proof}

As a consequence, we get the following.

\begin{corollary}
\cite{LevMan2014} If $G$\ is a unicyclic non-K\"{o}nig-Egerv\'{a}ry graph,
then $\mathrm{\ker}\left(  G\right)  =\mathrm{core}\left(  G\right)  $.
\end{corollary}

It is easy to see that for every non-negative integer $k$ there exits a graph
$G$ with $\left\vert \mathrm{core}(G)\right\vert =k$. For instance,
$\left\vert \mathrm{core}(K_{3})\right\vert =0$, while the graph $G$, obtained
from $K_{3}$ by joining $k\geq1$ leaves to one of the vertices of $K_{3}$, has
$\left\vert \mathrm{core}(G)\right\vert =k$.

\begin{proposition}
\label{prop11}\cite{levm3} If $G$ is a connected bipartite graph of order at
least two, then $\left\vert \mathrm{core}(G)\right\vert \neq1$.
\end{proposition}

\begin{corollary}
If $G$ is an almost bipartite non-K\"{o}nig-Egerv\'{a}ry graph, then
$\left\vert \mathrm{core}(G)\right\vert \neq1$.
\end{corollary}

\begin{proof}
Clearly, if $G=C_{2k+1}$, then $\mathrm{core}(G)=\emptyset$. If $G\neq
C_{2k+1}$, then, by Theorem \ref{th3322}, we have that
\[%
{\displaystyle\bigcup\limits_{y\in V(C)}}
\mathrm{core}\left(  D_{y}-y\right)  =\mathrm{core}\left(  G\right)  ,
\]
while by Proposition \ref{prop11}, we know that $\left\vert \mathrm{core}%
(D_{y}-y)\right\vert \neq1$ for each $y\in V(C)$, since $D_{y}-y$\ is
bipartite. Hence we finally get $\left\vert \mathrm{core}(G)\right\vert \neq1$.
\end{proof}

\begin{corollary}
\cite{LevMan2014} If $G$ is a unicyclic non-K\"{o}nig-Egerv\'{a}ry graph, then
$\left\vert \mathrm{core}(G)\right\vert \neq1$.
\end{corollary}

There exist non-bipartite K\"{o}nig-Egerv\'{a}ry graphs and
non-K\"{o}nig-Egerv\'{a}ry graphs that have $\left\vert \mathrm{core}%
(G)\right\vert =1$; e.g., the graph $G_{1}$ in Figure \ref{Fig4}\ and the
graphs in Figure \ref{Fig7}.

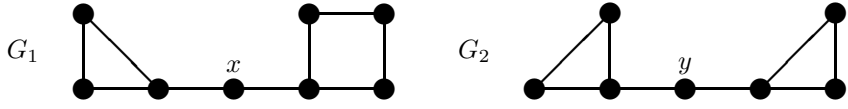
\begin{figure}[h]
\setlength{\unitlength}{1cm}\begin{picture}(5,1.3)\thicklines
\multiput(2,0)(1,0){5}{\circle*{0.29}}
\put(2,1){\circle*{0.29}}
\multiput(5,1)(1,0){2}{\circle*{0.29}}
\put(2,0){\line(1,0){4}}
\put(2,0){\line(0,1){1}}
\put(2,1){\line(1,-1){1}}
\put(5,0){\line(0,1){1}}
\put(5,1){\line(1,0){1}}
\put(6,0){\line(0,1){1}}
\put(4,0.3){\makebox(0,0){$x$}}
\put(1.2,0.5){\makebox(0,0){$G_{1}$}}
\multiput(8,0)(1,0){5}{\circle*{0.29}}
\multiput(9,1)(3,0){2}{\circle*{0.29}}
\put(8,0){\line(1,0){4}}
\put(8,0){\line(1,1){1}}
\put(9,0){\line(0,1){1}}
\put(11,0){\line(1,1){1}}
\put(12,0){\line(0,1){1}}
\put(10,0.3){\makebox(0,0){$y$}}
\put(7.2,0.5){\makebox(0,0){$G_{2}$}}
\end{picture}\caption{ \textrm{core}$(G_{1})=\left\{  x\right\}  $,
\textrm{core}$(G_{2})=\left\{  y\right\}  $ and only $G_{1}$ is a
K\"{o}nig-Egerv\'{a}ry graph}%
\label{Fig7}%
\end{figure}

It is worth noticing that there exists an almost bipartite
K\"{o}nig-Egerv\'{a}ry graph with a critical independent set meeting its
unique cycle. For instance, the bull graph.

There exist non-K\"{o}nig-Egerv\'{a}ry graphs satisfying $\mathrm{corona}%
\left(  G\right)  \cup$ $N(\mathrm{core}(G))\neq V(G)$; e.g., the graph in
Figure \ref{Fig9} has $\mathrm{corona}\left(  G\right)  \cup$ $N(\mathrm{core}%
(G))=V(G)-\{a\}$. \begin{figure}[h]
\setlength{\unitlength}{1cm}\begin{picture}(5,1.2)\thicklines
\multiput(4,0)(1,0){6}{\circle*{0.29}}
\put(4,1){\circle*{0.29}}
\multiput(6,1)(1,0){3}{\circle*{0.29}}
\put(4,0){\line(1,0){5}}
\put(4,0){\line(0,1){1}}
\put(4,1){\line(1,-1){1}}
\put(5,0){\line(1,1){1}}
\put(6,1){\line(1,0){1}}
\put(7,0){\line(0,1){1}}
\put(8,0){\line(0,1){1}}
\put(5,0.3){\makebox(0,0){$a$}}
\put(8.3,1){\makebox(0,0){$b$}}
\put(9.3,0){\makebox(0,0){$c$}}
\put(3.2,1){\makebox(0,0){$G$}}
\end{picture}\caption{$G$ is a non-K\"{o}nig-Egerv\'{a}ry graph with
\textrm{core}$(G)=\left\{  b,c\right\}  $}%
\label{Fig9}%
\end{figure}
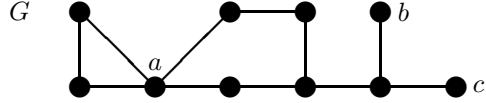

\begin{theorem}
\label{th333}If $G$ is an almost bipartite non-K\"{o}nig-Egerv\'{a}ry graph, then

\emph{(i)} $\mathrm{corona}\left(  G\right)  \cup$ $N(\mathrm{core}(G))=V(G)$;

\emph{(ii)} $\mathrm{corona}(G)$ $=V(C)\cup\left(
{\displaystyle\bigcup\limits_{y\in V(C)}}
\mathrm{corona}(D_{y}-y)\right)  $.
\end{theorem}

\begin{proof}
\emph{(i) }It is enough to show that $V(G)\subseteq$ $\mathrm{corona}(G)\cup$
$N(\mathrm{core}(G))$.

Let $a\in V(G)$.

\textit{Case 1}. $a\in V(C)$. If $b\in N(a)\cap V(C)$, then, by Lemma
\ref{lem2}\emph{(ii)}, the edge $ab$ is $\alpha$-critical. Hence
$a\in\mathrm{corona}(G)$.

\textit{Case 2}. $a\in V(G)-V(C)$. It follows that $a\in V(D_{y}-y)$, for some
$y\in V(C)$.

Since $G\left[  D_{y}-y\right]  $ is bipartite, by Theorem \ref{th1}%
\emph{(iii)}, we know that $V(D_{y}-y)=\mathrm{corona}(D_{y}-y)\cup$
$N(\mathrm{core}(D_{y}-y))$, while by Theorem \ref{th2}\emph{(iii)}, we have
that $\Omega\left(  G\right)  |_{V\left(  D_{y}-y\right)  }=\Omega\left(
D_{y}-y\right)  $ for every $y\in V\left(  C\right)  $, which ensures that
$\mathrm{corona}(D_{y}-y)\subseteq\mathrm{corona}(G)$.

Therefore, either $a\in\mathrm{corona}(D_{y}-y)\subseteq\mathrm{corona}(G)$,
or $a\in N(\mathrm{core}(D_{y}-y))\subseteq N(\mathrm{core}(G))$, because
$\mathrm{core}\left(  D_{y}-y\right)  \subseteq\mathrm{core}\left(  G\right)
$, by Theorem \ref{th2}\emph{(ii)}. Thus, $a\in\mathrm{corona}(G)\cup$
$N(\mathrm{core}(G))$.

All in all, $V(G)=\mathrm{corona}(G)\cup$ $N(\mathrm{core}(G))$.

\emph{(ii) }In\emph{ }the proof of Part \emph{(i)} we showed that
$\mathrm{corona}(D_{y}-y)\subseteq\mathrm{corona}(G)$ for every $y\in V\left(
C\right)  $, and $V\left(  C\right)  \subseteq\mathrm{corona}(G)$.

Hence, $V(C)\cup\left(
{\displaystyle\bigcup\limits_{y\in V(C)}}
\mathrm{corona}(D_{y}-y)\right)  \subseteq\mathrm{corona}(G)$. To complete the
proof, it remains to validate that $\mathrm{corona}(G)\subseteq V(C)\cup
\left(
{\displaystyle\bigcup\limits_{y\in V(C)}}
\mathrm{corona}(D_{y}-y)\right)  $. Let $a\in\mathrm{corona}(G)$. Then, $a\in
S$ for some $S\in\Omega\left(  G\right)  $. Suppose $a\notin V(C)$, \ then
there must be $y\in V(C)$ such that $a\in D_{y}-y$. Thus, $a\in S\cap
V(D_{y}-y)\subseteq\mathrm{corona}(D_{y}-y)$, because $\Omega\left(  G\right)
|_{V\left(  D_{y}-y\right)  }=\Omega\left(  D_{y}-y\right)  $, in accordance
with Theorem \ref{th2}\emph{(iii)}.
\end{proof}

\begin{theorem}
\cite{LevMan2020}\label{th44} If $G$ is an almost bipartite
non-K\"{o}nig-Egerv\'{a}ry graph, then
\[
d(G)=\alpha(G)-\mu(G)=\left\vert \mathrm{core}(G)\right\vert -\left\vert
N(\mathrm{core}(G))\right\vert .
\]

\end{theorem}

\begin{theorem}
\label{th5}If $G$ is an almost bipartite non-K\"{o}nig-Egerv\'{a}ry graph,
then%
\[
\left\vert \mathrm{corona}(G)\right\vert +\left\vert \mathrm{core}%
(G)\right\vert =2\alpha\left(  G\right)  +1.
\]

\end{theorem}

\begin{proof}
Let $S\in\Omega\left(  G\right)  $. According to Theorem \ref{th333} and Lemma
\ref{lem2}, we infer that
\[
\left\vert \mathrm{corona}(G)\right\vert +\left\vert N\left(  \mathrm{core}%
\left(  G\right)  \right)  \right\vert =\left\vert V\left(  G\right)
\right\vert =\alpha\left(  G\right)  +\mu\left(  G\right)  +1.
\]

By Theorem \ref{th44}, we obtain%

\begin{gather*}
\left\vert \mathrm{corona}(G)\right\vert +\left\vert \mathrm{core}\left(
G\right)  \right\vert =\left\vert \mathrm{corona}(G)\right\vert +\left\vert
N(\mathrm{core}(G))\right\vert +\alpha(G)-\mu(G)\\
=\alpha\left(  G\right)  +\mu\left(  G\right)  +1+\alpha(G)-\mu(G)=2\alpha
(G)+1
\end{gather*}
as required.
\end{proof}

\begin{corollary}
\cite{LevMan2014} If $G$ is a unicyclic non-K\"{o}nig-Egerv\'{a}ry graph, then
$\left\vert \mathrm{corona}(G)\right\vert +\left\vert \mathrm{core}%
(G)\right\vert =2\alpha\left(  G\right)  +1$.
\end{corollary}

\section{Conclusions}

It is known that for every graph $\mathrm{\ker}(G)\subseteq\mathrm{core}(G)$.
In this paper we showed that an almost bipartite non-K\"{o}nig-Egerv\'{a}ry
graph satisfies $\mathrm{\ker}(G)=\mathrm{core}(G)$, like bipartite graphs and
unicyclic non-K\"{o}nig-Egerv\'{a}ry graphs.

\begin{problem}
Characterize graphs enjoying $\mathrm{\ker}(G)=\mathrm{core}(G)$.
\end{problem}

We \ also proved that $\mathrm{corona}(G)\cup$ $N\left(  \mathrm{core}%
(G)\right)  =V(G)$ is true for almost bipartite non-K\"{o}nig-Egerv\'{a}ry
graphs, like for K\"{o}nig-Egerv\'{a}ry graphs.

\begin{problem}
Characterize graphs enjoying $\mathrm{corona}(G)\cup$ $N\left(  \mathrm{core}%
(G)\right)  =V(G)$.
\end{problem}

Theorem \ref{th5} claims that $\left\vert \mathrm{corona}(G)\right\vert
+\left\vert \mathrm{core}(G)\right\vert =2\alpha\left(  G\right)  +1$ holds
for almost bipartite non-K\"{o}nig-Egerv\'{a}ry graphs, like for unicyclic
non-K\"{o}nig-Egerv\'{a}ry graphs.

\begin{problem}
Characterize graphs enjoying $\left\vert \mathrm{corona}(G)\right\vert
+\left\vert \mathrm{core}(G)\right\vert =2\alpha\left(  G\right)  +1$.
\end{problem}

Proposition \ref{prop2}\ motivates the following.

\begin{conjecture}
\bigskip\label{conj1}If $G$ is an almost bipartite non-K\"{o}nig-Egerv\'{a}ry
graph, then every maximum matching of $G$ contains $\left\lfloor
\frac{\left\vert V(C)\right\vert }{2}\right\rfloor $ edges belonging to its
unique odd cycle $C$.
\end{conjecture}

\end{document}